\documentclass[11pt]{amsart}
\usepackage{amsmath}
\usepackage{amssymb}
\usepackage{amscd}
\usepackage{enumerate}
\usepackage{verbatim}
\usepackage{color}
\newtheorem{Thm}{Theorem}[section]
\newtheorem{Lemma}[Thm]{Lemma}
\newtheorem{Cor}[Thm]{Corollary}

\theoremstyle{definition}

\newtheorem{Rmk}[Thm]{Remark}

\def\bfe{\mathbf{e}}

\def\bfk{\mathbf{k}}
\def\bfx{\mathbf{x}}

\def\bfm{\mathbf{m}}

\def\bfp{\mathbf{p}}

\def\bbz{\mathbb{Z}}

\def\bbr{\mathbb{R}}
\def\bbz{\mathbb{Z}}
\def\bbn{\mathbb{N}}
\def\bbq{\mathbb{Q}}

\def\lra{\longrightarrow}

\def\x{\times}
\def\bs{\backslash}

\def\aut{\mathrm{Aut}}
\def\aff{\mathrm{Aff}}

\def\fix{\mathrm{Fix}}
\def\gfix{\mathrm{fix}}
\def\Nil{\mathrm{Nil}}

\def\R{(\mathrm{R})}
\def\Ad{\mathrm{Ad}}
\def\ad{\mathrm{ad}}

\def\GL{\mathrm{GL}}

\def\frakS{\mathfrak{S}}

\def\Sol{\mathrm{Sol}}
\def\GammaA{\Gamma_{\!A}}
\def\bfa{\mathbf{a}}

\def\calZ{\mathcal{Z}}

\def\nil{\mathrm{nil}}

\def\E{\mathrm{(E)}}
\def\R{\mathrm{(R)}}

\def\frakG{\mathfrak{G}}

\def\HPer{\mathrm{HPer}}
\def\DH{\mathrm{DH}}

\def\bbA{\mathbb{A}}

\def\bbD{\mathbb{D}}

\def\bfu{\mathbf{u}}

\def\boxit#1{\vbox{\hrule\hbox{\vrule\kern3pt
     \vbox{\kern3pt#1\kern3pt}\kern3pt\vrule}\hrule}}

\begin{document}
\title[Homotopy minimal periods of maps on infra-solvmanifolds]
{Density of the homotopy minimal periods of maps on infra-solvmanifolds}

\author{Jong Bum Lee}
\address{Department of Mathematics, Sogang University, Seoul 121-742, KOREA}
\email{jlee@sogang.ac.kr}

\author{Xuezhi Zhao}
\address{Department of Mathematics \& Institute of Mathematics and Interdisciplinary Science, Capital Normal University, Beijing 100048, CHINA}
\email{zhaoxve@mail.cnu.edu.cn}

\begin{comment}
\thanks{The second-named author is supported in part by Basic Science Researcher Program through the National Research Foundation of Korea(NRF) funded by the Ministry of Education (No. 2013R1A1A2058693) and by the Sogang University Research Grant of 2010 (10022)}
\end{comment}


\abstract
We study the homotopical minimal periods for maps on infra-solvmanifolds of type $\R$ using the density of the homotopical minimal period set in the natural numbers. This extends the result of \cite{Zhao et al} from flat manifolds to infra-solvmanifolds of type $\R$. Applying our main result we will list all possible maps on infra-solvmanifolds up to dimension three for which the corresponding density is positive.
\endabstract
\date{\today}

\maketitle

\tableofcontents

\section{Introduction}

Let $f:X\to X$ be a self-map on a topological space $X$. We define the following:
The set of \emph{periodic points} of $f$ with \emph{minimal period} $n$
\begin{align*}
P_n(f)&=\fix(f^n)-\bigcup_{k<n}\fix(f^k)
\end{align*}
and the set of \emph{homotopy minimal periods} of $f$
\begin{align*}
\HPer(f)&=\bigcap_{g \simeq f}\left\{n\in\bbn\mid P_n(g)\ne\emptyset\right\}.
\end{align*}
The famous \v{S}arkovs'ki\v{\i} theorem characterizes the dynamics (minimal periods) of a map of interval \cite{Sar}. The set of minimal periods of maps on the circle has been completely described in \cite{BGMY}. This led to a problem of study the set $\HPer(f)$ of homotopy minimal periods of $f$. Such an invariant gives an information about rigid dynamics of $f$. A fundamental question is to determine if the set $\HPer(f)$ is empty, finite or infinite.  This problem was successfully studied in \cite{JL} when the space $X$ is a torus of any dimension, and this was extended in \cite{JM} (see also \cite{JM-pacific, LZ-china}) to any nilmanifold, and in \cite{JKeM,LZ-AGT} to the special solvmanifolds modeled on $\Sol$.
When $X$ is the Klein bottle, the same problem was studied in \cite{KKZ}, and when $X$ is an infra-nilmanifold and $f$ is an expanding map, it was shown in \cite{Tau, LL-JGP, LZ-japan} that $\HPer(f)$ is co-finite.

It is now natural to seek for more information when $\HPer(f)$ becomes infinite. When $X$ is a flat manifold, some sufficient conditions on $X$ and $f$ for $\HPer(f)$ to be infinite were found in \cite{LZ, Zhao et al}. In particular, for this purpose, we will consider the following invariant, \cite[Definition~1.1]{Zhao et al} and \cite[Remark~3.1.60]{JM-book}: The \emph{lower density} of the homotopy minimal periods of $f$ is defined to be
$$
\DH(f)=\liminf_{n\to\infty}\frac{\#(\HPer(f)\cap[0,n])}{n}.
$$
From this definition, $\DH(f)\in[0,1]$. If $\HPer(f)$ is either empty or finite, then $\DH(f)=0$. So, we are interested in the case when $\HPer(f)$ is infinite. If one picks randomly a natural number, $\DH(f)$ is a lower bound for the probability of choosing number in $\HPer(f)$. Thus, the number $\DH(f)$ will bring to us more information about the periods of given map $f$ when $\HPer(f)$ is infinite.

The purpose of this paper is to generalize \cite{Zhao et al} from flat manifolds to infra-solvmanifolds of type $\R$ and thereby to list all possible maps on infra-solvmanifolds up to dimension three for which the corresponding density is positive.

\section{Infra-solvmanifolds}

Let $S$ be a connected and simply connected solvable Lie group. A discrete subgroup $\Gamma$ of $S$ is a {\it lattice} of $S$ if $\Gamma\bs{S}$ is compact, and in this case, we say that the quotient space $\Gamma\bs{S}$ is a {\it special} solvmanifold. Let $\Pi\subset\aff(S)$ be a torsion-free finite extension of the lattice $\Gamma=\Pi\cap S$ of $S$. That is, $\Pi$ fits the short exact sequence
$$
\CD
1@>>>S@>>>\aff(S)@>>>\aut(S)@>>>1\\
@.@AAA@AAA@AAA\\
1@>>>\Gamma@>>>\Pi@>>>\Pi/\Gamma@>>>1
\endCD
$$
Then $\Pi$ acts freely on $S$ and the manifold $\Pi\bs{S}$ is called an {\it infra-solvmanifold}. The finite group $\Phi=\Pi/\Gamma$ is the \emph{holonomy group} of $\Pi$ or $\Pi\bs{S}$. It sits naturally in $\aut(S)$. Thus every infra-solvmanifold $\Pi\bs{S}$ is finitely covered by the special solvmanifold $\Gamma\bs{S}$. An infra-solvmanifold $\Pi\bs{S}$ is of type $\R$ if $S$ is of \emph{type $\R$} or \emph{completely solvable}, i.e., if $\ad\, X:\frakS\to\frakS$ has only real eigenvalues for all $X$ in the Lie algebra $\frakS$ of $S$. It is known that if $S$ is of type $\R$, then it is of type $\E$, i.e., $\exp:\frakS\to S$ is surjective.

For any endomorphism $D$ on $S$, we denote the differential of $D:S\to S$ by $D_*:\frakS\to\frakS$.

Recall \cite[Theorem~1]{Wilking} in which it is proved that every infra-solvmanifold is modeled in a canonical way on a supersolvable Lie group. According to \cite[Lemma~2.3]{HL13-b} or \cite[Lemma~4.1]{Wilking}, a supersolvable Lie group is of type $\R$. Hence whenever we deal with infra-solvmanifolds, we may assume that we are given infra-solvmanifolds $M=\Pi\bs{S}$ of type $\R$.

We first recall the following:

\begin{Lemma}[{\cite[Lemma~2.1]{LL-Nagoya}}]\label{fully inv}
Let $S$ and $S'$ be simply connected solvable Lie groups, and let $\Pi\subset\aff(S)$ and $\Pi'\subset\aff(S')$ be finite extensions of lattices $\Gamma=\Pi\cap S$ of $S$ and $\Gamma'=\Pi'\cap S'$ of $S'$, respectively. Then there exist fully invariant subgroups $\Lambda\subset\Gamma$ and $\Lambda'\subset\Gamma'$ of $\Pi$ and $\Pi'$ respectively, which are of finite index, so that any homomorphism $\theta:\Pi\to\Pi'$ restricts to a homomorphism $\Lambda \to\Lambda'$.
\end{Lemma}

When the infra-solvmanifolds are of type $\R$, we have the following second Bieberbach type result.

\begin{Thm}[{\cite{HL13-b}}]\label{homotopy lift}
\begin{enumerate}
\item[$(1)$] Any continuous map $f:\Pi\bs{S}\to\Pi'\bs{S}'$ between infra-solvmanifolds of type $\R$ has an affine map $(d,D):S\to S'$ as a homotopy lift.
\item[$(2)$] Any continuous map $f:\Gamma\bs{S}\to\Gamma'\bs{S}'$ between special solvmanifolds of type $\R$ has a Lie group homomorphism $D:S\to S'$ as a homotopy lift.
\end{enumerate}
When $f$ is a homeomorphism, $D$ can be chosen to be invertible.
\end{Thm}

Let $f:\Pi\bs{S}\to\Pi\bs{S}$ be a self-map on the infra-solvmanifold $\Pi\bs{S}$ of type $\R$ with affine homotopy lift $(d,D):S\to S$. Since $\HPer(f)$ is a homotopy invariant, we may assume that $f$ is induced by the affine map $(d,D)$. The map $f$ induces a homomorphism $\varphi:\Pi\to\Pi$ on the group $\Pi$ of covering transformations of the covering projection $S\to\Pi\bs{S}$, which is given by
\begin{align}\label{homo}
\varphi(\alpha)(d,D)=(d,D)\alpha,\ \ \forall\alpha\in\Pi.\tag{$\ast$}
\end{align}
For any $(a,A)\in\Phi$, let $\varphi(a,A)=(a',A')$; then $A'D=DA$. Thus the homomorphism $\varphi$ induces a function $\bar\varphi:\Phi\to\Phi$ given by $\bar\varphi(A)=A'$ and this function satisfies $\bar\varphi(A)D=DA$ for all $A\in\Phi$. However, in general, $\bar\varphi$ is not necessarily a homomorphism.

Recall further that:
\begin{Thm}[{\cite[Theorem~6.1]{Je}}]\label{Je}
Let $f:M\to M$ be a self-map on a compact PL-manifold of dimension $\ge3$. Then $f$ is homotopic to a map $g$ with $P_n(g)=\emptyset$ if and only if $NP_n(f)=0$.
\end{Thm}

The infra-solvmanifolds of dimension $1$ or $2$ are the circle, the torus and the Klein bottle. Theorem~\ref{Je} for dimensions $1$ and $2$ is verified respectively in \cite{BGMY}, \cite{ABLSS} and \cite{Llibre, JKM, KKZ}. Immediately we have for any self-map $f$ on an infra-solvmanifold of any dimension,
\begin{align*}
\HPer(f)&=\{k\mid NP_k(f)\ne0\}.
\end{align*}
Recalling from \cite{Jiang} that
\begin{align*}
NP_n(f)&=\text{ (number of irreducible essential orbits of}\\
&\qquad\text{Reidemeister classes of $f^n$)}\x n,
\end{align*}
we have
\begin{align*}
\HPer(f)&=\{k\mid \exists\text{ an irreducible essential fixed point class of $f^k$}\}.
\end{align*}

Recall from {\cite[Propositions~9.1 and 9.3]{FL}} the following:
Let $f$ be a map on an infra-solvmanifold $\Pi\bs{S}$ of type $\E$ induced by an affine map $(d,D):S\to S$. For any $\alpha\in\Pi$, $\fix(\alpha(d,D))$ is an empty set or path connected. Hence every nonempty fixed point class of $f$ is path connected, and every isolated fixed point class forms an essential fixed point class with index
$\pm\det(I-A_*D_*)$ where $\alpha=(a,A)$. When the infra-solvmanifold $\Pi\bs{S}$ is of type $\R$, the converse also holds. Namely, every essential fixed point class of $f$ consists of a single element. Remark that $(d,D)^k$ induces the map $f^k$. Any fixed point class of $f^k$ is of the form $p(\fix(\alpha(d,D)^k))$ for some $\alpha=(a,A)\in\Pi$. It is essential if and only if it consists of a single element with index $\pm\det(I-A_*D_*^k)$.
%In this case, $\fix(\alpha(d,D)^k)=\{x\}$ for some $x\in S$.
Note further that it is reducible to $\ell$ if and only if $\ell\mid k$ and there exists $\beta\in\Pi$ such that $p(\fix(\beta(d,D)^\ell))\subset p(\fix(\alpha(d,D)^k))$, or equivalently, the Reidemeister class $[\beta]$ of $f^\ell$ is boosted up to the Reidemeister class $[\alpha]$ of $f^k$. This means that
$[\alpha]=[\beta\varphi^\ell(\beta)\varphi^{2\ell}(\beta)\cdots\varphi^{k-\ell}(\beta)]$ as the Reidemeister class of $f^k$. For some $\gamma\in\Pi$, we thus have $\alpha=\gamma(\beta\varphi^\ell(\beta)\varphi^{2\ell}(\beta)\cdots\varphi^{k-\ell}(\beta))\varphi^k(\gamma)^{-1}$.
Hence
\begin{align*}
\alpha&=(\gamma\beta\varphi^\ell(\gamma)^{-1})(\varphi^\ell(\gamma)\varphi^\ell(\beta)\varphi^{2\ell}(\gamma)^{-1})
\cdots(\varphi^{k-\ell}(\gamma)\varphi^{k-\ell}(\beta)\varphi^k(\gamma)^{-1})\\
&=\beta'\varphi^\ell(\beta')\varphi^{2\ell}(\beta')\cdots\varphi^{k-\ell}(\beta')
\end{align*}
with $\beta'=\gamma\beta\varphi^\ell(\gamma)^{-1}$. Consequently, the fixed point class $p(\fix(\alpha(d,D)^k))$ is irreducible if and only if for any $\beta\in\Pi$ and for any $\ell<k$ with $\ell\mid k$,
$$
\alpha(d,D)^k\ne(\beta(d,D)^\ell)^{k/\ell}
$$
or
$$
\alpha\ne\beta\varphi^\ell(\beta)\varphi^{2\ell}(\beta)\cdots\varphi^{k-\ell}(\beta).
$$

As a result, we can summarize the above observation as follows:

\begin{Thm}\label{T3.1}
Let $f:\Pi\bs{S}\to\Pi\bs{S}$ be a self-map on the infra-solvmanifold $\Pi\bs{S}$ of type $\E$ with an affine homotopy lift $(d,D):S\to S$. Let $\varphi:\Pi\to\Pi$ be the homomorphism induced by $(d,D)$, i.e., $\varphi(\alpha)(d,D)=(d,D)\alpha\ \forall\alpha\in\Pi$. Then
\begin{align*}
\HPer(f)&=\left\{k\ \Big|\
\begin{array}{l}
\exists \alpha=(a,A)\in\Pi \mathrm{\ such\ that}
\det(I-A_*D_*^k)\ne0\mathrm{\ and}\\
\forall\ \ell<k \mathrm{\ with\ } \ell\mid k,\ \forall\ \beta\in\Pi,\\
{\ }\qquad\qquad\alpha(d,D)^k\ne(\beta(d,D)^\ell)^{k/\ell}
\end{array}
\right\}\\
&=\left\{k\ \Big|\
\begin{array}{l}
\exists \alpha=(a,A)\in\Pi \mathrm{\ such\ that}
\det(I-A_*D_*^k)\ne0\mathrm{\ and}\\
\forall\ \ell<k \mathrm{\ with\ } \ell\mid k,\ \forall\ \beta\in\Pi,\\
{\ }\qquad\quad\alpha\ne\beta\varphi^\ell(\beta)\varphi^{2\ell}(\beta)\cdots\varphi^{k-\ell}(\beta)
\end{array}
\right\}.
\end{align*}
\end{Thm}

In order to generalize the results of \cite{Zhao et al} from flat manifolds to infra-solvmanifolds of type $\R$, we need the following observation which is crucial in our discussion.

\begin{Lemma}\label{rational matrix}
Let $\Lambda$ be a lattice of a connected, simply connected solvable Lie group $S$ {of type $\R$}, and let $K:S\to S$ be a Lie group homomorphism such that $K(\Lambda)\subset\Lambda$. For some choice of a linear basis in the Lie algebra $\frakS$ of $S$, $K_*$ is an upper block triangular matrix with diagonal blocks integer matrices; in particular $\det K_*$ is an integer.
\end{Lemma}

\begin{proof}
First we assume that $S$ is nilpotent and thus $\Lambda$ is a finitely generated torsion-free nilpotent group. The lower central series of $\Lambda$ is defined inductively via $\gamma_1(\Lambda)=\Lambda$ and $\gamma_{i+1}(\Lambda)=[\Lambda,\gamma_i(\Lambda)]$. The isolator of a subgroup $H$ of $\Lambda$ is defined by
$$
\sqrt[\Lambda]{H}=\{x\in\Lambda\mid x^k\in H \text{ for some $k\ge1$}\}.
$$
It is known (\cite[p.\!~473]{Passman}, \cite[Chap.\!~1]{Dekimpe} or \cite{KL-05}) that the sequence
$$
\Lambda=\Lambda_1\supset\Lambda_2=\sqrt[\Lambda]{\gamma_2(\Lambda)}\supset\cdots\supset
\Lambda_c=\sqrt[\Lambda]{\gamma_c(\Lambda)}\supset\Lambda_{c+1}=1
$$
forms a central series with $\Lambda_i/\Lambda_{i+1}\cong\bbz^{k_i}$. Now we can choose a generating set
$$
\bfa=\{\bfa_1,\cdots,\bfa_c\}
$$
in such a way that $\Lambda_i$ is the group generated by $\Lambda_{i+1}$ and $\bfa_i=\{a_{i1},\cdots,a_{in_i}\}$. We refer to $\bfa=\{\bfa_1,\cdots,\bfa_c\}$ as a {preferred basis} of $\Lambda$. Under the diffeomorphism $\log:S\to\frakS$, the image $\log\bfa$ of $\bfa$ is a basis of the vector space $\frakS$. Note also that  $\Lambda_i=\Lambda\cap{\gamma_i(S)}$ is a lattice of $\gamma_i(S)$ and a fully invariant subgroup of $\Lambda$. Since $K(\Lambda)\subset\Lambda$, it follows that $K(\Lambda_i)\subset\Lambda_i$ and the differential of $K$ is expressed as a rational matrix with respect to the basis $\log\bfa$ of the form
$$
\left[\begin{matrix}
      {K_c}_*&*&\cdots&*\\
      0&{K_{c-1}}_*&\cdots&*\\
      \vdots&\vdots&\ddots&\vdots\\
      0&0&\cdots&{K_1}_*
        \end{matrix}\right]
$$
where each square matrix ${K_i}_*$ is an integer matrix, see also \cite[Lemma~4.2]{LL-JGP}.

Now we go back to the cases where $S$ is solvable of type $\R$. According to \cite[Remark~8.2]{Wilking}, $\Lambda$ is a positive polycyclic group and $S$ is its supersolvable completion. Let $N$ be the maximal connected nilpotent normal subgroup of $S$. Then $\Lambda\cap N$ is the nilradical $\nil(\Lambda)$ of $\Lambda$, which is a lattice of $N$, see \cite[Proposition~5.1]{Wilking}. Hence we have the following diagram
$$
\CD
1@>>>N@>>>S@>>>S/N\cong\bbr^s@>>>1\\
@.@AAA@AAA@AAA\\
1@>>>\nil(\Lambda)@>>>\Lambda@>>>\Lambda/\nil(\Lambda)\cong\bbz^s@>>>1
\endCD
$$
By the assumption on $K$, $K$ restricts to a homomorphism $\kappa:\Lambda\to\Lambda$. Thus $\kappa$ and hence $K$ in turn restricts to $\kappa':\nil(\Lambda)\to\nil(\Lambda)$ and then induces a homomorphism $\bar\kappa:\bbz^s\to\bbz^s$. We choose a preferred basis of $\nil(\Lambda)$ under which $K':N\to N$ yields a rational matrix $K'_*$ with diagonal blocks integer matrices as above. Now we can complete the set of generators of $\nil(\Lambda)$ to a set of generators $\bfa=\{\bfa_0,\bfa_1,\cdots,\bfa_c\}$, called a \emph{preferred basis}, of $\Lambda$ so that $\bar\kappa$ induces an integer matrix $\bar{K}_*$ and so $\kappa$ induces an upper block triangular matrix
\begin{align*}
K_*&=\left[\begin{matrix}K'_*&*\\0&\bar{K}_*\end{matrix}\right]
\end{align*}
so that all diagonal blocks are integer matrices and hence $\det K_*$ is an integer.
\end{proof}

\begin{Rmk}\label{basis}
Let $\Lambda$ be a lattice of a connected, simply connected solvable Lie group $S$ of type $\R$. In the proof of the above lemma, we can choose a preferred basis (generator) $\bfa$ of $\Lambda$ so that $\log\bfa$ is a (linear) basis of the Lie algebra $\frakS$ of $S$ and if $K$ is a homomorphism on $S$ such that $K(\Lambda)\subset\Lambda$ then $K_*$ is an upper block triangular matrix with diagonal blocks integer matrices with respect to the ordered basis $\log\bfa$. We also refer to $\log\bfa$ as a \emph{preferred basis} of $\Lambda$ or $\frakS$.
\end{Rmk}

\section{Density of homotopy minimal periods}

In this section, we will generalize the main result of \cite{Zhao et al} from flat manifolds to infra-solvmanifolds of type $\R$.

Let $f$ be a self-map on an infra-solvmanifold $\Pi\bs{S}$ of type $\R$ with holonomy group $\Phi$. Let $f$ have an affine homotopy lift $(d,D)$. Recall that $f$ induces a homomorphism $\varphi:\Pi\to\Pi$ satisfying the identity \eqref{homo}: $\varphi(\alpha)(d,D)=(d,D)\alpha$, $\forall\alpha\in\Pi\subset S\rtimes\aut(S)$.
Let $\Gamma=\Pi\cap S$. It is not necessarily true that $\varphi(\Gamma)\subset\Gamma$. Using Lemma~\ref{fully inv}, we can choose a lattice $\Lambda\subset\Gamma$ of $S$ so that $\varphi(\Lambda)\subset\Lambda$. Thus for any $\lambda=(\lambda,I)\in\Lambda$, we have $\varphi(\lambda)=(\varphi(\lambda),I)$ and so
$$
(\varphi(\lambda),I)(d,D)=(d,D)(\lambda,I).
$$
Evaluating at the identity $1$ of $S$, we obtain that $\varphi(\lambda)\cdot d=d\cdot D(\lambda)$. Consequently, we have that
\begin{align*}%\label{homo'}
\varphi|_\Lambda=\mu(d)D.%\tag{$\ast\ast$}
\end{align*}
Furthermore, for any $(a,A)\in\Pi$, since $\Gamma$ is a normal subgroup of $\Pi$, we have $(a,A)(\gamma,I)(a,A)^{-1}\in\Gamma$; this implies $(\mu(a)A)(\Gamma)\subset\Gamma$ and so $(\mu(a)A)(\Lambda)\subset\Lambda$. Namely, we have homomorphisms $\mu(d)D, \mu(a)A:S\to S$ such that $(\mu(d)D)(\Lambda)\subset\Lambda$ and $(\mu(a)A)(\Lambda)\subset\Lambda$. We have to notice here that it is not necessary to have that $D(\Lambda), A(\Lambda)\subset\Lambda$. By Remark~\ref{basis}, we can choose a preferred basis $\bfa$ of $\Lambda$ so that $(\mu(d)D)_*=\Ad(d)D_*$ and $(\mu(a)A)_*=\Ad(a)A_*$ are upper block triangular rational matrices with diagonal blocks integer matrices with respect to the basis $\log\bfa$ of $\frakS$.

In what follows, we shall denote $\mu(d)D$ and $\mu(a)A$ by $\bbD$ and $\bbA$, respectively. As far as the determinants and the eigenvalues of the differentials of $\bbD$ and $\bbA$ are concerned, we may assume that they are \emph{integer matrices} $\bbD_*$ and $\bbA_*$. We can call them \emph{linearizations} of $D$, $(d,D)$ or $f$, and $\alpha=(a,A)\in\Pi$, respectively. We denote also the set of all integer linear combinations of the basis vectors in $\log\bfa$ by simply $\calZ$. Then we have $\bbD_*(\calZ)\subset\calZ$ and $\bbA_*(\calZ)\subset\calZ$.

In the following, we provide three lemmas that generalize \cite[Lemmas~4.3 and 4.5, Proposition~4.6]{Zhao et al} from flat manifolds to infra-solvmanifolds of type $\R$. These are essential in proving our main results.

\begin{Lemma}\label{L4.3}
Let $M=\Pi\bs{S}$ be an infra-solvmanifold of type $\R$ with holonomy group $\Phi$. Let $f$ be a self-map on $M$ with an affine homotopy lift $(d,D)$. Assume that
\begin{enumerate}
\item[$(1)$] any eigenvalue $\lambda$ of $\bbD_*$ of modulus $1$ is a root of unity, but not $1$;
\item[$(2)$] $\det \bbD_*\ne0,\pm1$.
\end{enumerate}
Then there exists a positive integer $N_0$ such that
$$
\Big|\det\left(\frac{I-\bbD_*^{k\ell}}{I-\bbD_*^\ell}\right)\Big|>1
$$
for all positive integers $k$ and $\ell$, provided their positive prime divisors are all greater than $N_0$.
\end{Lemma}

\begin{proof}
Let $\lambda_1,\cdots,\lambda_m$ be the eigenvalues of $\bbD_*$ counted with multiplicities. We first show that all $1-\lambda_i^\ell$ are nonzero. In fact, if $1-\lambda_i^\ell=0$ for some $i$; then $\lambda_i^\ell=1$ and so $\lambda_i$ is a primitive $\ell_0$-th root of unity for some $\ell_0$ where $1\le\ell_0\mid\ell$. Since $\lambda_i\ne1$, $\ell_0>1$. If $p$ is a prime divisor of $\ell_0$ then it is a prime divisor of $\ell$ and so $p>N_0$. {It follows that $[\bbq(\lambda_i):\bbq]\ge p-1>m$. This contradicts the fact that $[\bbq(\lambda_i):\bbq]$ is smaller than the size $m$ of $\bbD_*$.}
Thus $I-\bbD_*^\ell$ is invertible and $(I-\bbD_*^{k\ell})/(I-\bbD_*^\ell)=I+\bbD_*^\ell+\cdots+\bbD_*^{(k-1)\ell}$ and
$$
\det\left(\frac{I-\bbD_*^{k\ell}}{I-\bbD_*^\ell}\right)=\prod_{i=1}^m\frac{1-\lambda_i^{k\ell}}{1-\lambda_i^\ell}.
$$

Let $N_0>m+1$ and let $\ell$ be a positive integer all of whose prime divisors are greater than $N_0$.
Assume $|\lambda_i|=1$. By our assumption, $\lambda_i\ne1$ and is a root of unity. The above argument shows that $1-\lambda_i^\ell\ne0$ and by the same reasoning $1-\lambda_i^{k\ell}\ne0$; thus $(1-\lambda_i^{k\ell})/(1-\lambda_i^\ell)$ is nonzero and finite {for each such $k$ and $\ell$}. Hence we can choose a constant $\delta>0$ such that for all such $k$ and $\ell$
$$
\Big|\frac{1-\lambda_i^{k\ell}}{1-\lambda_i^\ell}\Big|>\delta.
$$
For $\lambda_i$ with $|\lambda_i|\ne1$, as $N_0\to\infty$, we have
$$
\Big|\frac{1-\lambda_i^{k\ell}}{1-\lambda_i^\ell}\Big|=|1+\lambda_i^\ell+\lambda_i^{2\ell}+\cdots+\lambda_i^{k-1}|\to
\begin{cases}1 &\text{when $|\lambda_i|<1$}\\
\infty &\text{when $|\lambda_i|>1$.}
\end{cases}
$$

By the assumption that $|\det \bbD_*|>1$, there exists an eigenvalue whose absolute value is bigger than $1$. Hence as $N_0\to\infty$ we have
$$
\prod_{i=1}^m\Big|\frac{1-\lambda_i^{k\ell}}{1-\lambda_i^\ell}\Big|\to\infty.
$$
Consequently, for $N_0$ large enough, the lemma is proved.
\end{proof}

\begin{Lemma}\label{L4.5}
Let $M=\Pi\bs{S}$ be an infra-solvmanifold of type $\R$ with holonomy group $\Phi$. Let $f$ be a self-map on $M$ with an affine homotopy lift $(d,D)$. Assume that
\begin{enumerate}
\item[$(1)$] any eigenvalue $\lambda$ of $\bbD_*$ of modulus $1$ is a root of unity, but not $1$;
\item[$(2)$] $\det \bbD_*\ne0,\pm1$.
\end{enumerate}
Then there exists a positive integer $N_1$ such that
$$
|\det(I-\bbD_*^{k})|>\sum_{1<\ell\!\!~\mid\!\!~k}|\det(I-\bbD_*^{k/\ell})|
$$
for all positive integers $k$, provided all its positive prime divisors are greater than $N_1$.
\end{Lemma}

\begin{proof}
Let $\lambda_1,\cdots,\lambda_m$ be the eigenvalues of $\bbD_*$ counted with multiplicities. From our assumptions and hence from the observations in the proof of Lemma~\ref{L4.3}, we have:
\begin{itemize}
\item Since $\det\bbD_*\ne0$, all $\lambda_i$ are nonzero.
\item If $|\lambda_i|=1$ then $\lambda_i\ne1$ and $\lambda_i$ is a root of unity, and $1-\lambda_i^k\ne0$ for all $k$ whose prime divisors are $>N_0$ where $N_0$ is a positive integer chosen in the previous lemma; hence there are constants $0<\delta_1<\delta_2$ such that for all $\lambda_i$ with $|\lambda_i|\le1$, we have $\delta_1\le|1-\lambda_i^k|\le\delta_2$ for all $k$ with this property.
\end{itemize}
For those eigenvalues with $|\lambda_i|>1$, we claim that there is a sufficiently large $k$ such that
$$
\sum_{1<\ell\!\!~\mid\!\!~k}|1-\lambda_i^{k/\ell}|<|1-\lambda_i^k|.
$$
Suppose on the contrary that for any $K>0$ there is $k_0>K$ such that
$$
|1-\lambda_i^{k_0}|\le\sum_{1<\ell\!\!~\mid\!\!~k_0}|1-\lambda_i^{k_0/\ell}|.
$$
Then
\begin{align*}
|1-\lambda_i^{k_0}|
\le \sum_{1<\ell\!\!~\mid\!\!~k_0}|1-\lambda_i^{k_0/2}|<\tau(k_0)|1-\lambda_i^{k_0/2}|
\end{align*}
where $\tau(k)$ is the number of all the divisors of $k$. {Since $\tau(k)\le2\sqrt{k}$ (see \cite[Exercise~3.2.17]{JM-book}),} we have
$$
2\sqrt{k_0}>|1+\lambda_i^{k_0/2}|\ge|\lambda_i|^{k_0/2}-1,
$$
which contradicts the obvious fact that $\lim_{k\to0}\sqrt{k}/(|\lambda_i|^{k/2}-1)=0$.

\begin{comment}
In fact, we observe first that for all sufficiently large $k$ we have the estimate
$$
\boxed{\sum_{1<\ell\!\!~\mid\!\!~k}|1-\lambda_i^{k/\ell}|\le\sum_{1<\ell\!\!~\mid\!\!~k} |1-\lambda_i^{k/2}|
\le\log_2k\ |1-\lambda_i^{k/2}|,??}
$$
because there are at most $\log_2k$ divisors of $k$ and $\ell=2$ is the smallest possible divisor.
\fbox{Xuezhi! Please clarify this.}
Note that for large $k$ we have $|1-\lambda_i^k|\sim|\lambda_i|^k$. Hence the inequality
\begin{align*}
|\lambda_i|^k\sim|1-\lambda_i^k|&\le \eta\sum_{1<\ell\!\!~\mid\!\!~k}|1-\lambda_i^{k/\ell}|\\
&\le \eta\log_2k\ |1-\lambda_i^{k/2}|\sim\eta\log_2k\ |\lambda_i|^{k/2}
\end{align*}
is impossible for any constant $\eta>0$ and for any sufficiently large $k$.
\end{comment}

Therefore we can choose $N_1\ge N_0$ such that if $k$ is a positive integer whose prime divisors are $\ge N_1$, then
\begin{align*}
\sum_{1<\ell\!\!~\mid\!\!~k}|\det(I-\bbD_*^{k/\ell})|
&=\sum_{1<\ell\!\!~\mid\!\!~k}\left(\prod_{i=1}^m|1-\lambda_i^{k/\ell}|\right)\\
&\le \prod_{i=1}^m\left(\sum_{1<\ell\!\!~\mid\!\!~k}|1-\lambda_i^{k/\ell}|\right)\\
&<\prod_{i=1}^m|1-\lambda_i^k|=|\det(I-\bbD_*^k)|.\qedhere
\end{align*}
\end{proof}

\begin{Lemma}\label{P4.6}
Let $M=\Pi\bs{S}$ be an infra-solvmanifold of type $\R$ with holonomy group $\Phi$. Let $f$ be a self-map on $M$ with an affine homotopy lift $(d,D)$. Assume that
\begin{enumerate}
\item[$(1)$] any eigenvalue $\lambda$ of $\bbD_*$ of modulus $1$ is a root of unity, but not $1$;
\item[$(2)$] $\det \bbD_*\ne0,\pm1$.
\end{enumerate}
Then there exists a positive integer $N_2$ such that the equality
$$
\calZ=\bigcup_{\substack{\ell\!\!~\mid\!\!~k\\ 1<\ell<k}} (I+\bbD_*^\ell+\bbD_*^{2\ell}+\cdots+\bbD_*^{k-\ell})(\calZ)
$$
is impossible for all positive integers $k$, provided its positive prime divisors are all greater than $N_2$.
\end{Lemma}

\begin{proof}
Remark that the proof of Lemma~\ref{L4.3} shows that there exists a positive integer $N_0$ such that for all positive integers $k$ whose prime divisors are greater than $N_0$, $I-\bbD_*^\ell$ has nonzero determinant if $\ell\mid k$. Since $\bbD_*(\calZ)\subset\calZ$, we have
\begin{align*}
(I-\bbD_*^k)(\calZ)&=(I+\bbD_*^\ell+\bbD_*^{2\ell}+\cdots+\bbD_*^{k-\ell})(I-\bbD_*^\ell)(\calZ)\\
&\subset(I+\bbD_*^\ell+\bbD_*^{2\ell}+\cdots+\bbD_*^{k-\ell})(\calZ)
\end{align*}
for all $k$ and $\ell$ with $\ell\mid k$. Thus if we had the equality
$$
\calZ=\bigcup_{\substack{\ell\!\!~\mid\!\!~k,1<\ell<k}}
(I+\bbD_*^\ell+\bbD_*^{2\ell}+\cdots+\bbD_*^{k-\ell})(\calZ)
$$
we would have
\begin{align*}
\calZ/(I-\bbD_*^k)(\calZ)
&=\left(\bigcup_{\ell\!\!~\mid\!\!~k, \ell\ne1,k} (I+\bbD_*^\ell+\bbD_*^{2\ell}+\cdots+\bbD_*^{k-\ell})(\calZ)\right)/(I-\bbD_*^k)(\calZ)\\
&=\bigcup_{\ell\!\!~\mid\!\!~k, \ell\ne1,k} \left((I+\bbD_*^\ell+\bbD_*^{2\ell}+\cdots+\bbD_*^{k-\ell})(\calZ)/(I-\bbD_*^k)(\calZ)\right)\\
&\cong \bigcup_{\ell\!\!~\mid\!\!~k, \ell\ne1,k} \left(\calZ/(I-\bbD_*^\ell)(\calZ)\right)
\end{align*}
and hence we would have
$$
|\det(I-\bbD_*^k)|\le \sum_{\ell\!\!~\mid\!\!~k, \ell\ne1,k} |\det(I-\bbD_*^\ell)|
\le\sum_{1<\ell\!\!~\mid\!\!~k} |\det(I-\bbD_*^{k/\ell})|.
$$
This contradicts Lemma~\ref{L4.5}.
\end{proof}

Now we are ready to state and prove our main results.

\begin{Thm}\label{T3.3}
Let $M=\Pi\bs{S}$ be an infra-solvmanifold of type $\R$ with holonomy group $\Phi$. Let $f$ be a self-map on $M$ with an affine homotopy lift $(d,D)$. Let $\varphi:\Pi\to\Pi$ be the homomorphism satisfying \eqref{homo}. Assume that
\begin{enumerate}
\item[$(1)$] any eigenvalue $\lambda$ of $\bbD_*$ of modulus $1$ is a root of unity, but not $1$;
\item[$(2)$] $\det \bbD_*\ne0,\pm1$;
\item[$(3)$] $\gfix(\bar\varphi:\Phi\to\Phi)=\{I\}$.
\end{enumerate}
Then there exists an integer $N$ with the following property: if $k$ is a positive integer with prime factorization $k=p_1^{n_1}\cdots p_s^{n_s}$ such that all $p_i$'s are greater than $N$, then $k\in\HPer(f)$.
\end{Thm}

\begin{proof}
Choose an integer $N$ so that $N\ge\max\{m+1,N_2,\ \text{order of $\bar\varphi$}\}$. Let $k=p_1^{n_1}\cdots p_s^{n_s}$ be a prime factorization of $k$ such that all $p_i$'s are greater than $N$. Then we have to show that $k\in\HPer(f)$. For this purpose, by Theorem~\ref{T3.1}, we need to find $\alpha=(a,A)\in\Pi$ satisfying:
\begin{itemize}
\item $\det(I-A_*D_*^k)\ne0$,
\item $\forall\ \ell<k \mathrm{\ with\ } \ell\mid k,\ \forall\ \beta\in\Pi,\ \
\alpha(d,D)^k\ne(\beta(d,D)^\ell)^{k/\ell}$.
\end{itemize}

We will show that we can choose $\alpha=(a,I)$ in $\Gamma\subset\Pi$. Recall first that the proof of Lemma~\ref{L4.3} shows that there exists a positive integer $N_0$ such that for all positive integers $k$ whose prime divisors are larger than $N_0$, $I-\bbD_*^\ell$ has nonzero determinant if $\ell\mid k$. Since $N\ge N_0$, $\det(I-\bbD_*^\ell)\ne0$ for all $\ell\mid k$. In particular, $\det(I-\bbD_*^k)\ne0$. By \cite[Lemma~3.3]{LL-Nagoya} or \cite[Theorem~3.3]{HL13-b}, we have $\det(I-D_*^k)=\det(I-\bbD_*^k)\ne0$.

It remains to prove the second condition. We assume on the contrary that for any $\alpha=(a,I)\in\Gamma$, there exists $\ell<k$ with $\ell\mid k$ and there exist $\beta=(b,B)\in\Pi$ such that $\alpha(d,D)^k=(\beta(d,D)^\ell)^{k/\ell}$, which is equivalent to
\begin{align}\label{ho}
\alpha=\beta\varphi^\ell(\beta)\varphi^{2\ell}(\beta)\cdots\varphi^{k-\ell}(\beta).\tag{$\dagger$}
\end{align}

Now we recall that since $D$ is an automorphism, $\varphi$ is the conjugation by $(d,D)$, $\varphi|_\Gamma=\mu(d)D$ and $\bar\varphi$ is the conjugation by $D$.
The matrix part (the holonomy part) of both sides of \eqref{ho} yields
$$
I=B\bar\varphi^\ell(B)\bar\varphi^{2\ell}(B)\cdots\bar\varphi^{k-\ell}(B).
$$
Taking $\bar\varphi^\ell$, we have
$$
I=\bar\varphi^\ell(B)\bar\varphi^{2\ell}(B)\cdots\bar\varphi^{k-\ell}(B)\bar\varphi^k(B).
$$
Hence
\begin{align*}
\bar\varphi^k(B)^{-1}=B^{-1}=\bar\varphi^\ell(B)\bar\varphi^{2\ell}(B)\cdots\bar\varphi^{k-\ell}(B).
\end{align*}
This gives us $\bar\varphi^k(B)=B$. By the choice of $k$, $k$ must be relatively prime to the order $p$ of $\bar\varphi$. Choose $x,y\in\bbz$ so that $kx+py=1$. Since $\bar\varphi=\bar\varphi^{kx+py}=(\bar\varphi^k)^x$, it follows that $\bar\varphi(B)=B$. Since $\gfix(\bar\varphi)=\{I\}$ by our assumption, we have $B=I$. Plugging into \eqref{ho}, we have
$$
a=b\varphi^\ell(b)\varphi^{2\ell}(b)\cdots\varphi^{k-\ell}(b).
$$
Since $\varphi|_\Gamma=\mu(d)D=\bbD$, we have
$$
a=b\bbD^\ell(b)\bbD^{2\ell}(b)\cdots\bbD^{k-\ell}(b)
$$
for some $\ell<k$ with $\ell\mid k$.

Now we have to show that for any $\ell<k$ with $\ell\mid k$
$$
\{e\bbD^\ell(e)\bbD^{2\ell}(e)\cdots\bbD^{k-\ell}(e)\mid e\in\Gamma\}\ne\Gamma.
$$
Recall in the proof of Lemma~\ref{rational matrix} that $\Gamma$ has a central series
$$
\Gamma=\Gamma_0\supset\nil(\Gamma)=\Gamma_1\supset\Gamma_2\supset\cdots\supset\Gamma_{c}\supset\Gamma_{c+1}=1
$$
with $\Gamma_i/\Gamma_{i+1}\cong\bbz^{k_i}$. Since $\bbD(\Gamma_i)\subset\Gamma_i$, it induces $\bar\bbD_i:\Gamma_i/\Gamma_{i+1}\to\Gamma_i/\Gamma_{i+1}$. Note also that
$$
\bbD_*=\left[\begin{matrix}\bar\bbD_c&*&\cdots&*\\
0&\bar\bbD_{c-1}&\cdots&*\\
\vdots&\vdots&\ddots&\vdots\\
0&0&\cdots&\bar\bbD_0\end{matrix}\right]
$$
where $\bar\bbD_i$ are integer matrices. Hence some $\bar\bbD_i$ satisfies the assumptions (1) and (2). By Lemma~\ref{P4.6}, we have
$$
\{\bar{e}+\bar\bbD_i^\ell(\bar{e})+\bar\bbD_i^{2\ell}(\bar{e})+\cdots+\bar\bbD_i^{k-\ell}(\bar{e})\mid \bar{e}\in\Gamma_i/\Gamma_{i+1}\}\ne\Gamma_i/\Gamma_{i+1}.
$$
This proves our assertion. Hence $a\in\Gamma$ can be chosen so that
$$
a\ne b\varphi^\ell(b)\varphi^{2\ell}(b)\cdots\varphi^{k-\ell}(b).
$$
for any $b\in\Gamma$. This contradiction proves the second condition.
\end{proof}

\begin{Cor}\label{T2.3}
Let $M=\Pi\bs{S}$ be an infra-solvmanifold of type $\R$ with holonomy group $\Phi$ and let $f$ be a self-map on $M$ with an affine homotopy lift $(d,D)$. Let $\varphi:\Pi\to\Pi$ be the homomorphism satisfying \eqref{homo}. Assume that
\begin{enumerate}
\item[$(1)$] any eigenvalue $\lambda$ of $\bbD_*$ of modulus $1$ is a root of unity, but not $1$;
\item[$(2)$] $\det \bbD_*\ne0,\pm1$;
\item[$(3)$] $\gfix(\bar\varphi:\Phi\to\Phi)=\{I\}$.
\end{enumerate}
Then $\DH(f)$ is positive.
\end{Cor}

\begin{proof}
By Theorem~\ref{T3.3}, there exists an integer $N$ with the following property: for any positive integer $k=p_1^{n_1}\cdots p_s^{n_s}$ with all $p_i$'s distinct primes and greater than $N$, $k\in\HPer(f)$. Thus
$$
\HPer(f)\supset\{k\mid \text{any prime divisor of $k$ is $>N$}\}.
$$
Let $q_1,\cdots q_\ell$ be the all prime numbers which are smaller than or equal to $N$. Then the set
$$
\{k\mid k\equiv1\!\!\mod{q_1\cdots q_\ell}\}
$$
is contained in the set on the right-hand side of the above. For, if $k\equiv1\!\!\mod{q_1\cdots q_\ell}$ and if $p$ is a prime divisor of $k$ with $p\le N$ then $p=q_j$ for some $j$; thus $q_j\mid k$ and $q_j\mid k-1$ and hence $q_j=1$, a contradiction.

Furthermore, we have that $N!\mid k-1$ implies $q_1\cdots q_\ell\mid k-1$. This shows that
$$
\HPer(f)\supset\{k\mid k\equiv1\!\!\mod{N!}\}
$$
and the set on the right-hand side has density $1/N!$. Consequently,
\begin{align*}
\DH(f)&\ge1/N!>0.\qedhere
\end{align*}
\end{proof}

A special solvmanifold is an infra-solvmanifold with the trivial holonomy group. Hence the third condition of Corollary~\ref{T2.3} on such a manifold is automatically fulfilled. Immediately we have:

\begin{Cor}\label{lattice}
Let $f$ be a self-map on a special solvmanifold $M$ with a Lie group homomorphism $D$ as a homotopy lift. Assume that
\begin{enumerate}
\item[$(1)$] any eigenvalue $\lambda$ of $\bbD_*$ of modulus $1$ is a root of unity, but not $1$;
\item[$(2)$] $\det \bbD_*\ne0,\pm1$;
\end{enumerate}
Then $\DH(f)$ is positive.
\end{Cor}

\section{Computational results}

In this section, applying our main result we will illustrate the method to determine the density of the homotopy minimal periods of all maps on infra-solvmanifolds up to dimension three.

For infra-solvmanifolds up to dimension $3$, there are only three possibilities for the solvable Lie group $G$ on which the manifold is modeled. It can be modeled on either the Abelian groups $\bbr^n (n\le3)$, the $2$-step nilpotent Heisenberg group $\Nil$ or the $2$-step solvable Lie group $\Sol$.

We can find a complete description of $\HPer(f)$ for maps $f$ on tori in \cite{BGMY} and \cite{JL}, and on the Klein bottle in \cite{KKZ}. In what follows we shall consider the remaining infra-solvmanifolds $\Pi\bs{G}$ of dimension $3$.

For any self-map $f$ on the manifold $\Pi\bs{G}$, let $\varphi:\Pi\to\Pi$ be a homomorphism induced by $f$.
Consider an affine map $(d,D)$ on $G$ satisfying \eqref{homo}. To apply Corollary~\ref{T2.3}, we have to consider the case where $\bbD=\mu(d)D$ is invertible. If this is the case, then \eqref{homo} says that $\varphi$ is the conjugation by $(d,D)$, that is, $\varphi(\alpha)=(d,D)\alpha(d,D)^{-1}$. If $\alpha=(a,I)\in\Gamma$ then
$$
\varphi(\alpha)=(d,D)(a,I)(d,D)^{-1}=(dD(a)d^{-1},I)=(\mu(d)D(a),I).
$$
Here $\mu(d)$ is the automorphism on $G$ obtained by conjugating by the element $d\in G$. Thus $\varphi(\Gamma)\subset\Gamma$ and $\varphi|_\Gamma=\mu(d)D=\bbD$, and hence {$\bar\varphi$ is the conjugation by $D$}. In particular $\bar\varphi:\Phi\to\Phi$ is an isomorphism.

We start with the following easy observation.
\begin{Lemma}\label{nontrivial fix}
Let $\Phi$ be a group isomorphic to $\bbz_2, \bbz_4$ or $\bbz_6$. If $\psi$ is an isomorphism on $\Phi$, then $\gfix(\psi)$ is not a trivial group.
\end{Lemma}

\begin{Lemma}\label{Z_2^2}
Let $\Phi$ be a group with presentation
$$
\Phi=\langle x,y\mid x^2=y^2=1, xy=yx\rangle,
$$
and let $\psi$ be an isomorphism on $\Phi$. Then $\gfix(\psi)=1$ if and only if $\psi$ satisfies one of the following:
\begin{itemize}
\item $\psi(x)=y$, $\psi(y)=xy$
\item $\psi(x)=xy$, $\psi(y)=x$
\end{itemize}
\end{Lemma}

\subsection{Flat manifolds of dimension three}

We have a complete classification of three-dimensional Bieberbach groups. There are six orientable ones and four nonorientable ones, see the book \cite[Theorems~3.5.5 and 3.5.9]{Wolf}. Every group has an explicit representation into $\bbr^3\rtimes \GL(4,\bbz)$ (not into $\bbr^4\rtimes O(4)$) in this book. Of course one of them is $\mathfrak{G}_1=\bbz^3$. Let
$$
e_1=\left[\begin{matrix}1\\0\\0\end{matrix}\right],\
e_2=\left[\begin{matrix}0\\1\\0\end{matrix}\right],\
e_3=\left[\begin{matrix}0\\0\\1\end{matrix}\right],
$$
and let $t_i=(e_i,I)\in\bbr^3\rtimes\GL(3,\bbz)$. Then $t_1,t_2$ and $t_3$ generate the subgroup $\Gamma$ of $\bbr^3\rtimes \GL(3,\bbz)$, which is isomorphic to the group of all integer vectors of $\bbr^3$.

The Bieberbach group $\mathfrak{G}_2$, $\mathfrak{G}_4$, $\frakG_5$, $\mathfrak{B}_1$ or $\mathfrak{B}_2$
has the holonomy group which is isomorphic to $\bbz_2,\bbz_4$ or $\bbz_6$.
By Lemma~\ref{nontrivial fix}, we cannot apply Corollary~\ref{T2.3}.

\subsubsection{The Bieberbach group $\mathfrak{G}_3$ $($\cite{Zhao et al}$)$}
Let $\alpha=(a,A)\in\bbr^3\rtimes\GL(3,\bbz)$ where
$$
a=\left[\begin{matrix}\frac{1}{3}\\0\\0\end{matrix}\right],\
A=\left[\begin{matrix}1&0&\hspace{8pt}0\\0&0&-1\\0&1&-1\end{matrix}\right].
$$
Then
$$
\mathfrak{G}_3=\langle t_1,t_2,t_3,\alpha\mid [t_i,t_j]=1, \alpha^3=t_1,  \alpha t_2\alpha^{-1}=t_3, \alpha t_3\alpha^{-1}=t_2^{-1}t_3^{-1}\rangle
$$
with the holonomy group $\Phi=\langle A\rangle\cong\bbz_3$.

Let $\varphi:\mathfrak{G}_3\to\mathfrak{G}_3$ be any homomorphism with invertible $\bbD$. Since $\bar\varphi:\Phi\to\Phi$ is going to be an isomorphism with $\gfix(\bar\varphi)=\{I\}$, we must have that $\bar\varphi(A)=A^2$. Thus $\varphi$ has the form
$$
\varphi(t_2)=t_1^{k_2}t_2^{m_2}t_3^{n_2},\
\varphi(t_3)=t_2^{k_3}t_2^{m_3}t_3^{n_3},\
\varphi(\alpha)=\alpha^{3k+2}t_2^{m}t_3^{n}.
$$
Because the relations $\alpha t_2\alpha^{-1}=t_3$ and $\alpha t_3\alpha^{-1}=t_2^{-1}t_3^{-1}$ are preserved by $\varphi$, it follows that
$k_2=k_3$ and $k_3=-k_2-k_3$; hence $k_2=k_3=0$, and $m_3=-m_2+n_2$ and $n_3=-m_2$. Hence
$$
\bbD_*=
\left[\begin{matrix}3\ell+2&0&0\\0&m_2&-m_2+n_2\\0&n_2&-m_2\end{matrix}\right].
$$
We claim that if $(m_2,n_2) \notin\{(0,0),(-1,-1),(\pm1,0),(0,\pm1),(1,1)\}$, then $\DH(f)>0$. First we observe that $\det\bbD_*=(3\ell+2)(m_2^2-m_2n_2+n_2^2)$ and $m_2^2-m_2n_2+n_2^2=\frac{1}{2}(m_2^2+(m_2-n_2)^2+n_2^2)\ge0$. If $\frac{1}{2}(m_2^2+(m_2-n_2)^2+n_2^2)<2$, we must have that either $|m_2|$ or $|n_2|$ is less than $2$. If $\frac{1}{2}(m_2^2+(m_2-n_2)^2+n_2^2)=0$ then $m_2=n_2=0$. If $\frac{1}{2}(m_2^2+(m_2-n_2)^2+n_2^2)=1$, then either $m_2=n_2=\pm1$, or $m_2=0$ and $n_2=\pm1$, or $m_2=\pm1$ and $n_2=0$. So, we have
\begin{align*}
\det\bbD_*=0 &\Leftrightarrow (m_2,n_2)=(0,0),\\
\det\bbD_*=\pm1 &\Leftrightarrow \det\bbD_*=-1\\
&\Leftrightarrow \ell=-1, (m_2,n_2) \in\{(-1,-1),(\pm1,0),(0,\pm1),(1,1)\}.
\end{align*}
Therefore if $(m_2,n_2) \notin\{(0,0),(-1,-1),(\pm1,0),(0,\pm1),(1,1)\}$, then $m_2^2-m_2n_2+n_2^2\ge2$ and $\det\bbD_*\ne0,\pm1$. Furthermore, the eigenvalues $\lambda$ of $\bbD_*$ are all real $3\ell+2$ and $\pm\sqrt{m_2^2-m_2n_2+n_2^2}$, in which $|\lambda|=1$ occurs only when $\ell=-1$ and in this case $\lambda=-1$. All the conditions of Corollary~\ref{T2.3} are satisfied. This proves our claim.

\subsubsection{The Bieberbach group $\mathfrak{G}_6$}
Let $\alpha=(a,A),\ \beta=(b,B)$ and $\gamma=(c,C)$ be elements of $\bbr^3\rtimes\GL(3,\bbz)$ where
\begin{align*}
&a=\left[\begin{matrix}\frac{1}{2}\\0\\0\end{matrix}\right],\
b=\left[\begin{matrix}0\\\frac{1}{2}\\\frac{1}{2}\end{matrix}\right],\
c=\left[\begin{matrix}\frac{1}{2}\\\frac{1}{2}\\\frac{1}{2}\end{matrix}\right],\\
&A=\left[\begin{matrix}1&\hspace{8pt}0&\hspace{8pt}0\\0&-1&\hspace{8pt}0\\0&\hspace{8pt}0&-1\end{matrix}\right],\
B=\left[\begin{matrix}-1&0&\hspace{8pt}0\\\hspace{8pt}0&1&\hspace{8pt}0\\\hspace{8pt}0&0&-1\end{matrix}\right],\
C=\left[\begin{matrix}-1&\hspace{8pt}0&0\\\hspace{8pt}0&-1&0\\\hspace{8pt}0&\hspace{8pt}0&1\end{matrix}\right].
\end{align*}
Then $A,B,C$ have order $2$ and $AB=C=BA$, and
$$
\mathfrak{G}_6=\left\langle{t_1,t_2,t_3,\alpha,\beta,\gamma\ \Big|
\begin{array}{l}
[t_i,t_j]=1, \gamma\beta\alpha=t_1t_3,\\
\alpha^2=t_1,  \alpha t_2\alpha^{-1}=t_2^{-1}, \alpha t_3\alpha^{-1}=t_3^{-1},\\
\beta t_1\beta^{-1}=t_1^{-1}, \beta^2=t_2, \beta t_3\beta^{-1}=t_3^{-1},\\
\gamma t_1\gamma^{-1}=t_1^{-1}, \gamma t_2\gamma^{-1}=t_2^{-1}, \gamma^2=t_3
\end{array}}\right\rangle.
$$
Thus $\mathfrak{G}_6$ fits the short exact sequence
$$
1\lra\Gamma\lra\mathfrak{G}_6\lra\Phi\lra1
$$
where $\Phi=\langle A,B\rangle\cong\bbz_2\oplus\bbz_2$. Every element of $\mathfrak{G}_6$ can be written uniquely in the form $\alpha^k\beta^mt_3^n$. We first observe the following:
Since $\gamma\beta\alpha=t_1t_3$, we have $\gamma=\alpha\beta^{-1}t_3$, and
$$
\beta^m\alpha^k=\begin{cases}
\alpha^k\beta^m&\text{when $(k,m)=(e,e)$}\\
\alpha^{-k}\beta^{m}&\text{when $(k,m)=(e,o)$}\\
\alpha^{k}\beta^{-m}&\text{when $(k,m)=(o,e)$}\\
\alpha^{-k}\beta^{-m}t_3&\text{when $(k,m)=(o,o)$}
\end{cases}
$$
and
\begin{align*}
(\alpha^k\beta^m t_3^n)^2&=\alpha^k(\beta^m\alpha^k)\beta^mt_3^{((-1)^{k+m}+1)n}\\
&=\begin{cases}
\alpha^{2k}\beta^{2m}t_3^{2n}&\text{when $(k,m)=(e,e)$}\\
\beta^{2m}&\text{when $(k,m)=(e,o)$}\\
\alpha^{2k}&\text{when $(k,m)=(o,e)$}\\
t_3^{2n-1}&\text{when $(k,m)=(o,o)$.}
\end{cases}
\end{align*}

Let $\varphi:\mathfrak{G}_6\to\mathfrak{G}_6$ be any homomorphism that induces an isomorphism $\bar\varphi$ on $\Phi$ satisfying  $\gfix(\bar\varphi)=\{I\}$. By Lemma~\ref{Z_2^2}, we have either $\bar\varphi(\bar\alpha)=\bar\beta$, $\bar\varphi(\bar\beta)=\bar\alpha\bar\beta$ or $\bar\varphi(\bar\alpha)=\bar\alpha\bar\beta$, $\bar\varphi(\bar\beta)=\bar\alpha$.

In general, $\varphi$ has the form
$$
\varphi(\alpha)=\alpha^{k_1}\beta^{m_1}t_3^{n_1},\
\varphi(\beta)=\alpha^{k_2}\beta^{m_2}t_3^{n_2},\
\varphi(t_3)=\alpha^{k_3}\beta^{m_3}t_3^{n_3}.
$$
Since $\gamma=\alpha\beta^{-1}t_3$, a simple calculation shows that
\begin{align*}
\varphi(\gamma)&=\alpha^{k_1}\beta^{m_1-m_2}\alpha^{-k_2+k_3}\beta^{m_3}t_3^{(-1)^{k_2+k_3+m_2+m_3}(n_1-n_2)+n_3}.
\end{align*}
\smallskip

\noindent
{\bf Case $\bar\varphi(\bar\alpha)=\bar\beta$, $\bar\varphi(\bar\beta)=\bar\alpha\bar\beta$.}\newline
Then $k_1$ is even and $k_2,m_1,m_2$ are odd. So, we have
\begin{align*}
\varphi(\gamma)
&=\alpha^{k_1-k_2+k_3}\beta^{-m_1+m_2+m_3}t_3^{n_1-n_2+n_3}.
\end{align*}
Since $\alpha^2=t_1,\beta^2=t_2$ and $\gamma^2=t_3$, a simple calculation shows that
\begin{align*}
&\alpha^2=t_1\Rightarrow \varphi(t_1)=\varphi(\alpha)^2=(\alpha^{k_1}\beta^{m_1}t_3^{n_1})^2=\beta^{2m_1}=t_2^{m_1};\\
&\beta^2=t_2\Rightarrow \varphi(t_2)=\varphi(\beta)^2=(\alpha^{k_2}\beta^{m_2}t_3^{n_2})^2=t_3^{2n_2-1};\\
&\gamma^2=t_3\Rightarrow k_3=2(k_1-k_2+k_3),m_3=0,n_3=0.
\end{align*}
Hence $\varphi(t_3)=\alpha^{k_3}=t_1^{k_3/2}$, and it follows that $\varphi(\Gamma)\subset\Gamma$ and so $\bbD=\varphi|_\Gamma$ and
$$
\bbD_*=\left[\begin{matrix}0&0&\frac{k_3}{2}\\m_1&0&0\\0&2n_2-1&0
\end{matrix}\right].
$$
Thus $\det\bbD_*=\frac{k_3}{2}m_1(2n_2-1)\ne0,\pm1$ if and only if either $k_3\ne0$ or $k_3=\pm2,m_1=\pm1,2n_2-1=\pm1$.
If $\bbD_*$ has an eigenvalue of modulus $1$, then $\det\bbD_*=\pm1$. This shows that the condition (2) of Corollary~\ref{T2.3} implies the condition (1). Consequently, when $k_1,k_3,m_3$ are even and $k_2,m_1,m_2$ are odd, if $k_3\ne0$ or
\begin{align*}
(k_3,m_1,n_2)\notin\{&(2,1,0),(2,-1,0),(-2,1,0),(-2,-1,0),\\
&(2,1,1),(2,-1,1),(-2,1,1),(-2,-1,1)\}
\end{align*}
then $\DH(f)>0$.
\smallskip

\noindent
{\bf Case $\bar\varphi(\bar\alpha)=\bar\alpha\bar\beta$, $\bar\varphi(\bar\beta)=\bar\alpha$.}\newline
Then $k_1,k_2,m_1$ are odd and $m_2$ is even. So, we have
\begin{align*}
\varphi(\gamma)
&=\alpha^{k_1+k_2-k_3}\beta^{-m_1+m_2+m_3}t_3^{-n_1+n_2+n_3}.
\end{align*}
Since $\alpha^2=t_1,\beta^2=t_2$ and $\gamma^2=t_3$, a simple calculation shows that
\begin{align*}
&\alpha^2=t_1\Rightarrow \varphi(t_1)=\varphi(\alpha)^2=(\alpha^{k_1}\beta^{m_1}t_3^{n_1})^2=t_3^{2n_1-1};\\
&\beta^2=t_2\Rightarrow \varphi(t_2)=\varphi(\beta)^2=(\alpha^{k_2}\beta^{m_2}t_3^{n_2})^2=\alpha^{2k_2}=t_1^{k_2};\\
&\gamma^2=t_3\Rightarrow k_3=0,m_3=2(-m_1+m_2+m_3),n_3=0.
\end{align*}
Hence $\varphi(t_3)=\beta^{m_3}=t_2^{m_3/2}$, and it follows that $\varphi(\Gamma)\subset\Gamma$ and so $\bbD=\varphi|_\Gamma$ and
$$
\bbD_*=\left[\begin{matrix}0&k_2&0\\0&0&\frac{m_3}{2}\\2n_1-1&0&0
\end{matrix}\right].
$$
Hence $\det\bbD_*=k_2\frac{m_3}{2}(2n_1-1)\ne0,\pm1$ if and only if either $m_3\ne0$ or $k_2=\pm1,m_3=\pm2,2n_1-1=\pm1$.
If $\bbD_*$ has an eigenvalue of modulus $1$, then $\det\bbD_*=\pm1$. This shows that the condition (2) of Corollary~\ref{T2.3} implies the condition (1). Consequently, when $k_1,k_2,m_1$ are odd and $k_3,m_2,m_3$ are even, if $m_3\ne0$ or
\begin{align*}
(k_3,m_1,n_2)\notin\{&(1,2,0),(1,-2,0),(-1,2,0),(-1,-2,0),\\
&(1,2,1),(1,-2,1),(-1,2,1),(-1,-2,1)\}
\end{align*}
then $\DH(f)>0$.

\subsubsection{The Bieberbach group $\mathfrak{B}_3$}
Let $\alpha=(a,A),\epsilon=(b,B)\in\bbr^3\rtimes\GL(3,\bbz)$ where
\begin{align*}
a=\left[\begin{matrix}\frac{1}{2}\\0\\0\end{matrix}\right],\
b=\left[\begin{matrix}0\\\frac{1}{2}\\0\end{matrix}\right],\
A=\left[\begin{matrix}1&\hspace{8pt}0&\hspace{8pt}0\\0&-1&\hspace{8pt}0\\0&\hspace{8pt}0&-1\end{matrix}\right],\
B=\left[\begin{matrix}1&0&\hspace{8pt}0\\0&1&\hspace{8pt}0\\0&0&-1\end{matrix}\right].
\end{align*}
Then $A,B$ have order $2$ and $AB=BA$, and
$$
\mathfrak{B}_3=\left\langle{t_1,t_2,t_3,\alpha,\epsilon\ \Big|
\begin{array}{l}
[t_i,t_j]=1, \epsilon\alpha\epsilon^{-1}=t_2\alpha,\\
\alpha^2=t_1,  \alpha t_2\alpha^{-1}=t_2^{-1}, \alpha t_3\alpha^{-1}=t_3^{-1},\\
\epsilon t_1\epsilon^{-1}=t_1, \epsilon^2=t_2, \epsilon t_3\epsilon^{-1}=t_3^{-1}
\end{array}}\right\rangle.
$$
Thus $\mathfrak{B}_3$ fits the short exact sequence
$$
1\lra\Gamma\lra\mathfrak{B}_3\lra\Phi\lra1
$$
where $\Phi=\langle A,B\rangle\cong\bbz_2\oplus\bbz_2$. Every element of $\mathfrak{B}_3$ can be written uniquely in the form $\alpha^k\epsilon^mt_3^n$. We first observe the following:
From $\epsilon\alpha\epsilon^{-1}=t_2\alpha$, we have
\begin{align*}
&\epsilon^m\alpha^k=\alpha^k\epsilon^{(-1)^km},\\
&(\alpha^k\epsilon^mt_3^n)^2=\alpha^{2k}\epsilon^{((-1)^k+1)m}t_3^{((-1)^{k+m}+1)n}.
\end{align*}

Let $\varphi:\mathfrak{B}_3\to\mathfrak{B}_3$ be any homomorphism. Then $\varphi$ has the form
$$
\varphi(\alpha)=\alpha^{k_1}\epsilon^{m_1}t_3^{n_1},\
\varphi(\epsilon)=\alpha^{k_2}\epsilon^{m_2}t_3^{n_2},\
\varphi(t_3)=\alpha^{k_3}\epsilon^{m_3}t_3^{n_3}.
$$
Assume that
%Let $(d,D)$ be an affine map satisfying \eqref{homo} so that $\bbD=\mu(d)D$ is invertible. Then
$\varphi$ induces an isomorphism $\bar\varphi:\Phi\to\Phi$ so that $\gfix(\bar\varphi)=\{I\}$. Write $\bar\alpha, \bar\epsilon$ for the images of $\alpha,\epsilon$ under $\mathfrak{B}_3\to\Phi$. Then
$$
\bar\varphi(\bar\alpha)=\bar\alpha^{k_1}\bar\epsilon^{m_1},\
\bar\varphi(\bar\epsilon)=\bar\alpha^{k_2}\bar\epsilon^{m_2},
$$
and {$k_3,m_3$ must be even}, i.e., $\varphi(t_3)\in\Gamma$. Due to $\gfix(\bar\varphi)=\{I\}$, $\bar\varphi$ must be one of the forms given in Lemma~\ref{Z_2^2}.
\medskip

\noindent
{\bf Case $\bar\varphi(\bar\alpha)=\bar\epsilon$, $\bar\varphi(\bar\epsilon)=\bar\alpha\bar\epsilon$.}\newline
Then $k_1$ is even and $k_2,m_1,m_2$ are odd. So, we have
\begin{align*}
&\varphi(t_1)=\varphi(\alpha)^2=t_1^{k_1}t_2^{m_1};\\
&\varphi(t_2)=\varphi(\epsilon)^2=t_1^{k_2}t_3^{2n_2};\\
&\varphi(t_3)=\alpha^{k_3}\epsilon^{m_3}t_3^{n_3}=t_1^{k_3/2}t_2^{m_3/2}t_3^{n_3}.
\end{align*}
Hence it follows that $\varphi(\Gamma)\subset\Gamma$ and so $\bbD=\varphi|_\Gamma$ and
$$
\bbD_*=\left[\begin{matrix}k_1&k_2&\frac{k_3}{2}\\m_1&0&\frac{m_3}{2}\\0&2n_2&n_3
\end{matrix}\right].
$$
On the other hand,
\begin{align*}
&\epsilon\alpha\epsilon^{-1}=t_2\alpha\Rightarrow \alpha\epsilon^{-1}=\epsilon\alpha\Rightarrow k_2=m_1=0,\\
&\alpha t_2\alpha^{-1}=t_2^{-1}\Rightarrow k_2=0,\\
&\alpha t_3\alpha^{-1}=t_3^{-1}\Rightarrow k_3=m_3=0.
\end{align*}
These induce that $\det\bbD_*=0$.
\smallskip

\noindent
{\bf Case $\bar\varphi(\bar\alpha)=\bar\sigma\bar\epsilon$, $\bar\varphi(\bar\epsilon)=\bar\alpha$.}\newline
Then $k_1,k_2,m_1$ are odd and $m_2$ is even. So, we have
\begin{align*}
&\epsilon\alpha\epsilon^{-1}=t_2\alpha\Rightarrow \alpha\epsilon^{-1}=\epsilon\alpha\Rightarrow k_2=0,m_1=m_2,n_1=0.
\end{align*}
and hence
\begin{align*}
&\varphi(t_2)=\varphi(\epsilon)^2=t_1^{k_2}=1.
\end{align*}
This induces that $\det\bbD_*=0$.
\smallskip

In all, Corollary~\ref{T2.3} is not applicable for the flat manifold $\mathfrak{B}_3\bs\bbr^3$.

\subsubsection{The Bieberbach group $\mathfrak{B}_4$}
Let $\alpha=(a,A),\epsilon=(b,B)\in\bbr^3\rtimes\GL(3,\bbz)$ where
\begin{align*}
a=\left[\begin{matrix}\frac{1}{2}\\0\\0\end{matrix}\right],\
b=\left[\begin{matrix}0\\\frac{1}{2}\\\frac{1}{2}\end{matrix}\right],\
A=\left[\begin{matrix}1&\hspace{8pt}0&\hspace{8pt}0\\0&-1&\hspace{8pt}0\\0&\hspace{8pt}0&-1\end{matrix}\right],\
B=\left[\begin{matrix}1&0&\hspace{8pt}0\\0&1&\hspace{8pt}0\\0&0&-1\end{matrix}\right].
\end{align*}
Then $A,B$ have order $2$ and $AB=BA$, and
$$
\mathfrak{B}_4=\left\langle{t_1,t_2,t_3,\alpha,\epsilon\ \Big|
\begin{array}{l}
[t_i,t_j]=1, \epsilon\alpha\epsilon^{-1}=t_2t_3\alpha,\\
\alpha^2=t_1,  \alpha t_2\alpha^{-1}=t_2^{-1}, \alpha t_3\alpha^{-1}=t_3^{-1},\\
\epsilon t_1\epsilon^{-1}=t_1, \epsilon^2=t_2, \epsilon t_3\epsilon^{-1}=t_3^{-1}
\end{array}}\right\rangle.
$$
Thus $\mathfrak{B}_4$ fits the short exact sequence
$$
1\lra\Gamma\lra\mathfrak{B}_4\lra\Phi\lra1
$$
where $\Phi=\langle A,B\rangle\cong\bbz_2\oplus\bbz_2$. Every element of $\mathfrak{B}_4$ can be written uniquely in the form $\alpha^k\epsilon^mt_3^n$. We first observe the following:
From $\epsilon\alpha\epsilon^{-1}=t_2t_3\alpha$, we have
\begin{align*}
&\epsilon^m\alpha^k=\alpha^k\epsilon^{(-1)^km}t_3^{\frac{1-(-1)^k}{2}\frac{1-(-1)^{m}}{2}},\\
&(\alpha^k\epsilon^mt_3^n)^2=\alpha^{2k}\epsilon^{((-1)^k+1)m}t_3^{((-1)^{k+m}+1)n-\frac{1-(-1)^k}{2}\frac{1-(-1)^{m}}{2}}.
\end{align*}

Let $\varphi:\mathfrak{B}_4\to\mathfrak{B}_4$ be any homomorphism. Then $\varphi$ has the form
$$
\varphi(\alpha)=\alpha^{k_1}\epsilon^{m_1}t_3^{n_1},\
\varphi(\epsilon)=\alpha^{k_2}\epsilon^{m_2}t_3^{n_2},\
\varphi(t_3)=\alpha^{k_3}\epsilon^{m_3}t_3^{n_3}.
$$
Assume that
%Let $(d,D)$ be an affine map satisfying \eqref{homo} so that $\bbD=\mu(d)D$ is invertible. Then
$\varphi$ induces an isomorphism $\bar\varphi:\Phi\to\Phi$ so that $\gfix(\bar\varphi)=\{I\}$. Write $\bar\alpha, \bar\epsilon$ for the images of $\alpha,\epsilon$ under $\mathfrak{B}_3\to\Phi$. Then
$$
\bar\varphi(\bar\alpha)=\bar\alpha^{k_1}\bar\epsilon^{m_1},\
\bar\varphi(\bar\epsilon)=\bar\alpha^{k_2}\bar\epsilon^{m_2},
$$
and {$k_3,m_3$ must be even}, i.e., $\varphi(t_3)\in\Gamma$. Due to $\gfix(\bar\varphi)=\{I\}$, $\bar\varphi$ must be one of the forms given in Lemma~\ref{Z_2^2}.
\medskip

\noindent
{\bf Case $\bar\varphi(\bar\alpha)=\bar\epsilon$, $\bar\varphi(\bar\epsilon)=\bar\alpha\bar\epsilon$.}\newline
Then $k_1$ is even and $k_2,m_1,m_2$ are odd. So, we have
\begin{align*}
&\epsilon\alpha\epsilon^{-1}=t_2t_3\alpha\Rightarrow \alpha\epsilon^{-1}=\epsilon t_3\alpha\Rightarrow k_3=-2k_2, m_3=2m_1, n_3=0,\\
&\alpha t_3\alpha^{-1}=t_3^{-1}\Rightarrow k_3=m_3=0.
\end{align*}
Hence $\varphi(t_3)=1$. This induces $\det\bbD_*=0$.
\smallskip

\noindent
{\bf Case $\bar\varphi(\bar\alpha)=\bar\sigma\bar\epsilon$, $\bar\varphi(\bar\epsilon)=\bar\alpha$.}\newline
Then $k_1,k_2,m_1$ are odd and $m_2$ is even. So, we have $\alpha t_3\alpha^{-1}=t_3^{-1}\Rightarrow n_3=0$, but
$\epsilon\alpha\epsilon^{-1}=t_2t_3\alpha$ induces that $n_3=1-2n_1$ is odd, a contradiction.
\medskip

In all, Corollary~\ref{T2.3} is not applicable for the flat manifold $\mathfrak{B}_3\bs\bbr^3$.

\subsection{Infra-nilmanifolds modeled on $\Nil$}

We will consider the three-dimensional infra-nilmanifolds modeled on the Heisenberg group $\Nil$. Recall that
$$
\Nil=\left\{\left[\begin{matrix}1&x&z\\0&1&y\\0&0&1\end{matrix}\right]\ \Big|\ x,y,z\in\bbr\right\}.
$$
For all integers $k>0$, we consider the subgroups $\Gamma_k$ of $\Nil$:
$$
\Gamma_k=\left\{\left[\begin{matrix}1&m&-\frac{\ell}{k}\\0&1&\hspace{8pt}n\\0&0&\hspace{8pt}1\end{matrix}\right]\ \Big|\ \ell,m,n\in\bbz\right\}.
$$
These are lattices of $\Nil$ and every lattice of $\Nil$ is isomorphic to some $\Gamma_k$. Letting
$$
s_1=\left[\begin{matrix}1&1&0\\0&1&0\\0&0&1\end{matrix}\right],\
s_2=\left[\begin{matrix}1&0&0\\0&1&1\\0&0&1\end{matrix}\right],\
s_3=\left[\begin{matrix}1&0&-\frac{1}{k}\\0&1&\hspace{8pt}0\\0&0&\hspace{8pt}1\end{matrix}\right],
$$
we obtain a presentation of $\Gamma_k$
$$
\Gamma_k=\langle s_1,s_2,s_3\mid [s_3,s_1]=[s_3,s_2]=1, [s_2,s_1]=s_3^k\rangle.
$$
Every element of $\Gamma_k$ can be written uniquely as the form
$$
s_2^ns_1^ms_3^\ell=\left[\begin{matrix}1&m&-\frac{\ell}{k}\\0&1&\hspace{8pt}n\\0&0&\hspace{8pt}1\end{matrix}\right].
$$
Remark that $s_1^ms_2^n=s_2^ns_1^ms_3^{-kmn}$.
All possible almost-Bieberbach groups can be found in \cite[pp.~\!799--801]{Shin} or \cite{DIKL}. We will use the same names for our groups as in \cite{DIKL}.

The $3$-dimensional almost Bieberbach groups of type II, IV, X and XVI have holonomy groups $\bbz_2$ or $\bbz_6$, and so by Lemma~\ref{nontrivial fix} we cannot apply Corollary~\ref{T2.3}.

\subsubsection{Groups of type {\rm I}, $k>0$}
Let $f$ be a self-map on the nilmanifold $\Gamma_k\bs\Nil$. It can be seen easily, see for example \cite[p.~\!355]{LZ-china}, that all the possible (integer) matrices $\bbD_*$ are of the form
$$
\bbD_*=\left[\begin{matrix}
\alpha\delta-\beta\gamma&0&0\\
0&\alpha&\beta\\0&\gamma&\delta
\end{matrix}\right].
$$
Hence the eigenvalues of $\bbD_*$ are $\lambda_1,\lambda_2$ and $\lambda_1\lambda_2$ where
$$
\lambda_i=\frac{(\alpha+\delta)\pm\sqrt{(\alpha-\delta)^2+4\beta\gamma}}{2}
=\frac{(\lambda_1+\lambda_2)\pm|\lambda_1-\lambda_2|}{2}.
$$
Since $\det\bbD_*=\lambda_1^2\lambda_2^2$, we have $\det\bbD_*=0,\pm1 \Leftrightarrow \lambda_1\lambda_2=0,\pm1$. By the condition (2) of Corollary~\ref{lattice}, we must have $|\lambda_1\lambda_2|\ge2$. Note that the eigenvalues are all real. By the condition (1) of Corollary~\ref{lattice}, if $|\lambda|=1$ then $\lambda\ne1$ and $\lambda$ is a root of unity; hence $\lambda=-1$. In this case Corollary~\ref{lattice} implies $\DH(f)>0$. If we let $\ell=\lambda_1+\lambda_2=-\lambda_j$, then $\lambda_1+\lambda_2=-1-\ell$. This is exactly a part of the case~(4) of \cite[Theorem~5.5]{LZ-china}, which yields $\HPer(f)=\bbn-2\bbn$ and so $\DH(f)=1/2$. If the condition (1) of Corollary~\ref{lattice} is vacuously true, i.e., if $|\lambda_i|\ne1$ for $i=1,2$, then by Corollary~\ref{lattice} we have $\DH(f)>0$. Note that this is exactly the case~(5) of \cite[Theorem~5.5]{LZ-china}, which yields $\HPer(f)=\bbn$ and so $\DH(f)=1$.

\subsubsection{Groups of type {\rm VIII}, $k>0,k\equiv0\pmod4$}
Consider another almost Bieberbach group $\Pi$ given by
$$
\Pi=\left\langle{s_1, s_2, s_3, \alpha,\beta\ \Big|
\begin{array}{l}
[s_3,s_1]=[s_3,s_2]=1,\ [s_2,s_1]=s_3^k,\\
\alpha s_1\alpha^{-1}=s_1^{-1}s_3^{k/2},\ \alpha s_2\alpha^{-1}=s_2^{-1}s_3^{-k/2},\ \alpha^2=s_3,\\
\beta^2=s_1, \beta s_2\beta^{-1}=s_2^{-1}s_3^{k/2},\ \beta s_3\beta^{-1}=s_3^{-1},\\
\text{$[\alpha,\beta]=s_1^{-1}s_2^{-1}s_3^{k/2+1}$}
\end{array}}\right\rangle.
$$
This is a $3$-dimensional almost Bieberbach group $\pi_{4}$ with Seifert bundle type $4$, and has the holonomy group $\Phi=\bbz_2\oplus\bbz_2$ with generators $\bar\alpha$ and $\bar\beta$.

Let $\varphi:\Pi\to\Pi$ be a homomorphism that induces an isomorphism $\bar\varphi:\Phi\to\Phi$ satisfying $\gfix(\bar\varphi)=\{I\}$. Every element of $\Pi$ can be written uniquely as $s_2^n\beta^m\alpha^\ell$. By Lemma~\ref{Z_2^2}, we need to consider the following two cases.
\medskip

\noindent
{\bf Case $\bar\varphi(\bar\alpha)=\bar\beta$, $\bar\varphi(\bar\beta)=\bar\alpha\bar\beta$.}\newline
Then $\varphi$ has the form
$$
\varphi(s_2)=s_2^{n_2}s_1^{m_2}s_3^{\ell_2},\
\varphi(\beta)=s_2^{n_1}\beta^{2m_1+1}\alpha^{2\ell_1+1},\
\varphi(\alpha)=s_2^{n_3}s_1^{m_3}\alpha^{2\ell_3+1}.
$$
So, we have
\begin{align*}
\varphi(s_1)&=\varphi(\beta)^2=s_2^{2n_1+1}s_3^{-\frac{km_1}{2}(2n_1+1)},\\
\varphi(s_3)&=\varphi(\alpha)^2=(s_2^{n_3}s_1^{m_3}\alpha^{2\ell_3+1})^2
=s_3^{\frac{k}{2}(2m_3n_3+m_3-n_3)+2\ell_3+1}.
\end{align*}
On the other hand, a simple computation shows that the relation $\beta s_2\beta^{-1}=s_2^{-1}s_3^{k/2}$ induces
\begin{align*}
&n_2=0,\ \frac{k}{2}(2m_3n_3+m_3-n_3)+2\ell_3+1=-m_2(2n_1+1)
\end{align*}
These describe $\bbD=\mu(d)D=\varphi|_\Gamma$. It is now easy to see that the ``integral" differential of $\bbD$ with respect to the basis $\{\log(s_1),\log(s_2),\log(s_3)\}$ of $\mathfrak{nil}$ is
$$
\bbD_*=\left[\begin{matrix}
0&m_2&0\\2n_1+1&n_2&0\\0&0&-m_2(2n_1+1)
\end{matrix}\right].
$$

Remark that the eigenvalues $\lambda$ of $\bbD_*$ are $-m_2(2n_1+1)$ and the roots of the equation $x^2-m_2(2n_1+1)=0$. Thus $\lambda=-m_2(2n_1+1)$ or $\pm\sqrt{m_2(2n_1+1)}$.
Hence if $|\lambda|=1$ then one of the eigenvalues must be $1$. Consequently, the condition (1) of Corollary~\ref{T2.3} holds if and only if $m_2(2n_1+1)\ne\pm1$.  Note also that since $\det\bbD_*=m_2^2(2n_1+1)^2$, $\det\bbD_*\ne0,\pm1$ if and only if $m_2(2n_1+1)\ne0, \pm1$.

In all, we have shown that $\DH(f)>0$ if $m_2(2n_1+1)\ne0,\pm1$, or equivalently if
$$
(m_2,n_1)\notin\{(0,\bbz),(1,0),(-1,0),(1,-1),(-1,-1)\}.
$$
\smallskip

\noindent
{\bf Case $\bar\varphi(\bar\alpha)=\bar\alpha\bar\beta$, $\bar\varphi(\bar\beta)=\bar\alpha$.}\newline
Then $\varphi$ has the form
$$
\varphi(s_2)=s_2^{n_2}s_1^{m_2}s_3^{\ell_2},\
\varphi(\beta)=s_2^{n_1}s_1^{m_1}\alpha^{2\ell_1+1},\
\varphi(\alpha)=s_2^{n_3}\beta^{2m_3+1}\alpha^{2\ell_3+1}.
$$
So, we have
\begin{align*}
\varphi(s_1)&=\varphi(\beta)^2=s_3^{\frac{k}{2}(2m_1n_1+m_1-n_1)+2\ell_1+1}:=s_3^\ell,\\
\varphi(s_3)&=\varphi(\alpha)^2=s_2^{2n_3+1}s_3^{-\frac{km_3}{2}(2n_3+1)}.
\end{align*}
On the other hand, a simple computation shows that
\begin{align*}
&\alpha s_2\alpha^{-1}=s_2^{-1}s_3^{-k/2}\Rightarrow n_2=0,\ \ell=m_2(2n_3+1)
\end{align*}
These describe $\bbD=\mu(d)D=\varphi|_\Gamma$. It is now easy to see that the ``integral" differential of $\bbD$ with respect to the basis $\{\log(s_1),\log(s_2),\log(s_3)\}$ of $\mathfrak{nil}$ is
$$
\bbD_*=\left[\begin{matrix}
0&m_2&0\\0&n_2&2n_3+1\\m_2(2n_3+1)&0&0
\end{matrix}\right].
$$

Remark that the characteristic equation of $\bbD_*$ is $x^3-m_2^2(2n_3+1)^2=0$. Thus if an eigenvalue of $\bbD_*$ has modulus $1$, then one of the eigenvalues must be $1$. Consequently, the condition (1) of Corollary~\ref{T2.3} holds if and only if $m_2(2n_3+1)\ne\pm1$.  Note also that since $\det\bbD_*=m_2^2(2n_3+1)^2$, $\det\bbD_*\ne0,\pm1$ if and only if $m_2(2n_3+1)\ne0, \pm1$.

In all, we have shown that $\DH(f)>0$ if $m_2(2n_3+1)\ne0,\pm1$, or equivalently if
$$
(m_2,n_3)\notin\{(0,\bbz),(1,0),(-1,0),(1,-1),(-1,-1)\}.
$$

\subsubsection{Groups of type {\rm XIII}, $k>0,k\equiv0 \text{ or }2\pmod3$}
Consider the almost Bieberbach group $\Pi$ given by
$$
\Pi=\left\langle{s_1, s_2, s_3, \alpha \ \Big|
\begin{array}{l}
[s_3,s_1]=[s_3,s_2]=1,\ [s_2,s_1]=s_3^k,\\ \alpha s_1\alpha^{-1}=s_2,\
\alpha s_2\alpha^{-1}=s_1^{-1}s_2^{-1},\ \alpha^3=s_3
\end{array}}\right\rangle.
$$
This is a $3$-dimensional almost Bieberbach group $\pi_{6,1}$ or $\pi_{6,4}$ with Seifert bundle type $6$.

Let $\varphi:\Pi\to\Pi$ be a homomorphism. Every element of $\Pi$ is of the form $s_2^ns_1^ms_3^\ell$, $s_2^ns_1^ms_3^\ell\alpha$ or $s_2^ns_1^ms_3^\ell\alpha^2$. In order to have an isomorphism $\bar\varphi:\Phi\to\Phi$ such that $\gfix(\varphi)=\{I\}$, we must have that $\bar\varphi(\bar\alpha)=\bar\alpha^2$. This implies that $\varphi$ has the form
$$
\varphi(s_1)=s_2^{n_1}s_1^{m_1}s_3^{\ell_1},\
\varphi(s_2)=s_2^{n_2}s_1^{m_2}s_3^{\ell_2},\
\varphi(\alpha)=s_2^{n_3}s_1^{m_3}\alpha^{3\ell_3+2}.
$$
Then we can show that
$$
\varphi(s_3)=\varphi(\alpha)^3=s_3^{(3\ell_3+2)-\frac{m_3(m_3+1)}{2}k+(m_3^2+m_3n_3+n_3^2)k-\frac{n_3(n_3+1)}{2}k}
:=s_3^m.
$$
Furthermore, the relations $\alpha s_1\alpha^{-1}=s_2$ and $\alpha s_2\alpha^{-1}=s_1^{-1}s_2^{-1}$ are preserved by $\varphi$. This induces the conditions $n_1=n_2=-m_1$ and $m_2=-2m_1$. The relation $[s_2,s_1]=s_3^k$ yields that $m=-3m_1^2$. Consequently, the integral differential of $\bbD=\varphi|_{\Gamma_k}$ with respect to the basis $\{\log(s_1),\log(s_2),\log(s_3)\}$ of $\mathfrak{nil}$ is
$$
\bbD_*=\left[\begin{matrix}
\hspace{8pt}m_1&-2m_1&0\\-m_1&\hspace{4pt}-m_1&0\\\hspace{8pt}0&\hspace{8pt}0&-3m_1^2
\end{matrix}\right].
$$
Hence $\det\bbD_*=(3m_1^2)^2$ and the eigenvalues of $\bbD_*$ are $-3m_1^2$ and $\pm\sqrt{3}m_1$. No eigenvalues of $\bbD_*$ are of modulus $1$, and $\det\bbD_*=0$ (i.e., $m_1=0$) or $\det\bbD_*\ge9$. Consequently if $m_1\ne0$ then $\DH(f)>0$.

\subsubsection{Groups of type {\rm XIII}, $k>0,k\equiv0\text{ or }1\pmod3$}
Consider another almost Bieberbach group $\Pi$ given by
$$
\Pi=\left\langle{s_1, s_2, s_3, \alpha \ \Big|
\begin{array}{l}
[s_3,s_1]=[s_3,s_2]=1,\ [s_2,s_1]=s_3^k,\\ \alpha s_1\alpha^{-1}=s_2,\
\alpha s_2\alpha^{-1}=s_1^{-1}s_2^{-1},\ \alpha^3=s_3^2
\end{array}}\right\rangle.
$$
This is a $3$-dimensional almost Bieberbach group $\pi_{6,2}$ or $\pi_{6,3}$ with Seifert bundle type $6$.

Let $\varphi:\Pi\to\Pi$ be a homomorphism. Every element of $\Pi$ is of the form $s_2^ns_1^ms_3^\ell$, $s_2^ns_1^ms_3^\ell\alpha$ or $s_2^ns_1^ms_3^\ell\alpha^2$. In order to have an isomorphism $\bar\varphi:\Phi\to\Phi$ such that $\gfix(\varphi)=\{I\}$, we must have that $\bar\varphi(\bar\alpha)=\bar\alpha^2$. This implies that $\varphi$ has the form
$$
\varphi(s_1)=s_2^{n_1}s_1^{m_1}s_3^{\ell_1},\
\varphi(s_2)=s_2^{n_2}s_1^{m_2}s_3^{\ell_2},\
\varphi(\alpha)=s_2^{n_3}s_1^{m_3}\alpha^{3\ell_3+2}.
$$
Then it can be seen as before that
$$
\varphi(s_3^2)=\varphi(\alpha)^3=s_3^{(3\ell_3+2)-\frac{m_3(m_3+1)}{2}k+(m_3^2+m_3n_3+n_3^2)k-\frac{n_3(n_3+1)}{2}k}.
$$
Since $\varphi(s_3)\in\Gamma_k$, $\varphi(s_3)$ is of the form $s_2^ns_1^ms_3^\ell$ and so
$$
\varphi(s_3^2)=(s_2^ns_1^ms_3^\ell)^2=s_2^{2n}s_1^{2m}s_3^{2\ell-kmn}.
$$
Hence $\varphi(s_3)=s_3^\ell$. Furthermore, the relations $\alpha s_1\alpha^{-1}=s_2$ and $\alpha s_2\alpha^{-1}=s_1^{-1}s_2^{-1}$ are preserved by $\varphi$. This induces the conditions $n_1=n_2=-m_2$ and $m_1=-2m_2$. The relation $[s_2,s_1]=s_3^k$ yields that $\ell=m_1n_2-m_2n_1=3m_2^2$. Consequently, the integral differential of $\bbD=\varphi|_{\Gamma_k}$ with respect to the basis $\{\log(s_1),\log(s_2),\log(s_3)\}$ of $\mathfrak{nil}$ is
$$
\bbD_*=\left[\begin{matrix}
-2m_2&\hspace{8pt}m_2&0\\\hspace{4pt}-m_2&-m_2&0\\\hspace{8pt}0&\hspace{8pt}0&3m_2^2
\end{matrix}\right].
$$
Hence $\det\bbD_*=(3m_2^2)^2$ and the eigenvalues of $\bbD_*$ are $\ell$ and $\frac{-3\pm\sqrt{3}i}{2}m_2$. No eigenvalues of $\bbD_*$ are of modulus $1$, and $\det\bbD_*=0$ (i.e., $m_2=0$) or $\det\bbD_*\ge9$. Consequently if $m_2\ne0$ then $\DH(f)>0$.

\subsection{Infra-solvmanifolds modeled on $\Sol$}

Next we will consider a closed $3$-manifold with $\Sol$-geometry. Recall that $\Sol=\bbr^2\rtimes_\phi\bbr$ where
$$
\phi(t)=\left[\begin{matrix}e^t&0\\0&e^{-t}\end{matrix}\right].
$$
Then $\Sol$ is a connected and simply connected unimodular $2$-step solvable Lie group of type $\R$.
It has a faithful representation into $\aff(\bbr^3)$ as follows:
$$
\Sol=\left\{\left[\begin{matrix}e^t&0&0&x\\0&e^{-t}&0&y\\0&0&1&t\\0&0&0&1\end{matrix}\right]\Big|\ x,y,t\in\bbr\right\}.
$$

Let $M$ be a closed $3$-manifold with $\Sol$-geometry. Then the fundamental group $\Pi$ of $M$ is a Bieberbach group of $\Sol$, and $M=\Pi\bs\Sol$. Further, $\Pi$ can be embedded into $\aff(\Sol)=\Sol\rtimes\aut(\Sol)$ so that there is an exact sequence
$$
1\lra\Gamma\lra\Pi\lra\Pi/\Gamma\lra1
$$
where $\Gamma=\Pi\cap\Sol$ is a lattice of $\Sol$ and $\Phi=\Pi/\Gamma$ is a finite group, called the holonomy group of $\Pi$ or $M$, which sits naturally into $\aut(\Sol)$, see \cite{HL13}. The lattices $\Gamma$ of $\Sol$ are determined by $2\x2$-integer matrices $A$
$$
A=\left[\begin{matrix}\ell_{11}&\ell_{12}\\\ell_{21}&\ell_{22}\end{matrix}\right]
$$
of determinant $1$ and trace $>2$, see for example \cite[Lemma~2.1]{LZ}. Namely,
$$
\Gamma=\GammaA=\langle a_1,a_2,\tau \mid[a_1,a_2]=1, \tau a_i\tau^{-1}=A(a_i)\rangle=\bbz^2\rtimes_A\bbz.
$$
It is known from \cite[Corollary~8.3]{HL13} that $\Pi$ is isomorphic to one of the following groups, adopting the naming in \cite{HL13-p} and using the notation $\bfa^\bfx$ for the elements in $\Gamma$ of the form $a_1^{x}a_2^{y}$:
\begin{enumerate}
\item[$(1)$]
$\Pi_1=\GammaA=\left\langle{ a_1,a_2,\tau \mid [a_1,a_2]=1, \tau a_i\tau^{-1}=A(a_i)}\right\rangle$.
\item[$(2)$]
$\Pi_2^\pm=\left\langle{a_1,a_2,\sigma \mid [a_1,a_2]=1, \sigma a_i\sigma^{-1}=N_\pm(a_i)}\right\rangle$\newline
where $N_\pm$ are square roots of $A$. The holonomy group is $\bbz_2$.
\item[$(3)$]
$\Pi_3=\left\langle{ a_1, a_2, \tau, \alpha \ \Big|
\begin{array}{l}
[a_1,a_2]=1,\ \tau a_i\tau^{-1}=A(a_i),\\
\alpha a_i\alpha^{-1}=M(a_i),\\
\alpha^2=\bfa^{\bfm},\ \alpha\tau\alpha^{-1}=\bfa^{\bfk'}\tau^{-1}
\end{array}
}\right\rangle$\newline
for some integer matrix $M$. The holonomy group is $\bbz_2$.
\item[$(4)$]
$\Pi_6=\left\langle{ a_1, a_2, \sigma, \alpha \ \Big|
\begin{array}{l}
[a_1,a_2]=1,\ \sigma a_i\sigma^{-1}=N(a_i),\\ \alpha a_i\alpha^{-1}=M(a_i),\\
\alpha^2=\bfa^{\bfm},\ \alpha \sigma\alpha^{-1}=\bfa^{\bfk'}\sigma^{-1}
\end{array}}\right\rangle$\newline
where $N$ is a square matrix of $A$ and $M$ is some integer matrix. The holonomy group is $\bbz_2\oplus\bbz_2$.
\end{enumerate}

Due to Lemma~\ref{nontrivial fix}, the manifolds with fundamental group $\Pi_2^\pm$ or $\Pi_3$ are not our concern and we will focus on the special solvmanifold $\GammaA\bs\Sol$ and the manifold with fundamental group $\Pi_6$.

\subsubsection{The special solvmanifold $\GammaA\bs\Sol$}

Let $f$ be a self-map on $\GammaA\bs\Sol$. By \cite[Theorem 2.4]{LZ}, the homomorphism $\varphi:\GammaA\to\GammaA$ induced by $f$ is determined by
$$
\varphi(a_i)=\bfa^{\bfu_i},\ \varphi(\tau)=\bfa^\bfp\tau^\zeta
$$
for some $\bfu_i,\bfp\in\bbz^2$ and $\zeta\in\bbz$. Note that $\varphi$ extends uniquely to a Lie group homomorphism on $\Sol$. It follows easily that all the possible (integer) matrices $\bbD_*$ are of the form
$$
\bbD_*=\left[\begin{matrix}\bfu_1&\bfu_2&{\bf0}\\
0&0&\zeta\end{matrix}\right].
$$
We say that $\varphi$ is of type (I) if $\zeta=1$; of type (II) if $\zeta=-1$; of type (III) if $\zeta\ne\pm1$. When $\varphi$ is of type (III), we have $\varphi(a_i)=1$.

Now we consider the conditions of Corollary~\ref{lattice}. These eliminate $\varphi$ of type (I) and (III). If $\varphi$ of type (II) satisfies the conditions of Corollary~\ref{lattice}, then $\DH(f)>0$. In fact, it is shown in \cite[Theorem~5.1]{LZ-AGT} that such a map has $\HPer(f)=\bbn-2\bbn$, and so $\DH(f)=1/2$.

\subsubsection{The infra-solvmanifold $\Pi_6\bs\Sol$}

Consider the Bieberbach groups given by
$$
\Pi_6=\left\langle{ a_1, a_2, \sigma, \alpha \ \Big|
\begin{array}{l}
[a_1,a_2]=1,\ \sigma a_i\sigma^{-1}=N(a_i),\\ \alpha a_i\alpha^{-1}=M(a_i),\\
\alpha^2=\bfa^{\bfe_2},\ \alpha \sigma\alpha^{-1}=\bfa^{\bfk'}\sigma^{-1}
\end{array}}\right\rangle.
$$
Let $\varphi:\Pi_6\to\Pi_6$ be a homomorphism and let $\tau=\sigma^2$. Then the subgroup $\langle a_1,a_2,\tau\rangle$ of $\Pi_6$ is a lattice of $\Sol$ and is a fully invariant subgroup of $\Pi_6$. Thus every element of $\Pi_6$ is of the form $\bfa^\bfx\tau^r$, $\bfa^\bfx\tau^r\sigma$, $\bfa^\bfx\tau^r\alpha$ or $\bfa^\bfx\tau^r\sigma\alpha$, and
$$
\varphi(a_i)=\bfa^{\bfu_i},\ \varphi(\tau)=\bfa^\bfp\tau^\zeta.
$$
Let $(d,D)$ be an affine map on $\Sol$ satisfying the identity \eqref{homo}. Then $\bbD=\mu(d)D=\varphi|_{\langle a_1,a_2,\tau\rangle}$ and so
$$
\bbD_*=\left[\begin{matrix}\bfu_1&\bfu_2&0\\0&0&\zeta\end{matrix}\right].
$$

Since $\Phi_6\cong\bbz_2\oplus\bbz_2$ is generated by $\bar\sigma$ and $\bar\alpha$, $\bar\varphi$ must be one of the forms given in Lemma~\ref{Z_2^2}. Thus we have $\bar\varphi(\bar\sigma)=\bar\alpha$ or $\bar\varphi(\bar\sigma)=\bar\sigma\bar\alpha$. This implies that $\varphi(\sigma)$ is of the form $\bfa^\bfx\tau^r\alpha$ or $\bfa^\bfx\tau^r\sigma\alpha$. The identity $\sigma^2=\tau$ induces in either case $\zeta=0$ and consequently $\det\bbD_*=0$. Hence we cannot apply Corollary~\ref{T2.3}.

\end{document}